\documentclass{amsart} 
\usepackage{amssymb}
\usepackage{amsmath}
\usepackage{amsfonts}

\begin{document}
\newtheorem{theo}{Theorem}[section]
\newtheorem{atheo}{Theorem*}
\newtheorem{prop}[theo]{Proposition}
\newtheorem{aprop}[atheo]{Proposition*}
\newtheorem{lemma}[theo]{Lemma}
\newtheorem{alemma}[atheo]{Lemma*}
\newtheorem{exam}[theo]{Example}
\newtheorem{coro}[theo]{Corollary}
\theoremstyle{definition}
\newtheorem{defi}[theo]{Definition}
\newtheorem{rem}[theo]{Remark}


\newcommand{\Bb}{{\bf B}}
\newcommand{\Cb}{{\mathbb C}}
\newcommand{\Nb}{{\mathbb N}}
\newcommand{\Qb}{{\mathbb Q}}
\newcommand{\Rb}{{\mathbb R}}
\newcommand{\Zb}{{\mathbb Z}}
\newcommand{\Ac}{{\mathcal A}}
\newcommand{\Bc}{{\mathcal B}}
\newcommand{\Cc}{{\mathcal C}}
\newcommand{\Dc}{{\mathcal D}}
\newcommand{\Fc}{{\mathcal F}}
\newcommand{\Ic}{{\mathcal I}}
\newcommand{\Jc}{{\mathcal J}}
\newcommand{\Kc}{{\mathcal K}}
\newcommand{\Lc}{{\mathcal L}}
\newcommand{\Oc}{{\mathcal O}}
\newcommand{\Pc}{{\mathcal P}}
\newcommand{\Sc}{{\mathcal S}}
\newcommand{\Tc}{{\mathcal T}}
\newcommand{\Uc}{{\mathcal U}}
\newcommand{\Vc}{{\mathcal V}}

\author{Nik Weaver}

\title [Quantum graphs as quantum relations]
       {Quantum graphs as quantum relations}

\address {Department of Mathematics\\
Washington University\\
Saint Louis, MO 63130}

\email {nweaver@math.wustl.edu}

\date{\em June 14, 2017}

\thanks{Partially supported by NSF grant DMS-1067726.}


\begin{abstract}
The ``noncommutative graphs'' which arise in quantum error correction are a
special case of the quantum relations introduced in \cite{W}. We use this
perspective to interpret the Knill-Laflamme error-correction conditions
\cite{KL} in terms of graph-theoretic independence, to give intrinsic
characterizations of Stahlke's noncommutative graph homomorphisms \cite{S}
and Duan, Severini, and Winter's noncommutative bipartite graphs \cite{DSW},
and to realize the noncommutative confusability graph associated to a
quantum channel \cite{DSW} as the pullback of a diagonal relation.

Our framework includes as special cases not only purely classical and purely
quantum information theory, but also the ``mixed'' setting which arises in
quantum systems obeying superselection rules. Thus we are able to define
noncommutative confusability graphs, give error correction conditions, and
so on, for such systems. This could have practical value, as superselection
constraints on information encoding can be physically realistic.
\end{abstract}

\maketitle

\section{Quantum graphs}\label{sec1}

``Quantum'' or ``noncommutative'' graphs appear in the setting of quantum error
correction \cite{DSW, SS, S}. The usual construction starts with a quantum
channel, which in the Schrodinger picture is modelled by a completely positive
trace preserving (CPTP) map $\Phi: M_m \to M_n$. Here $M_m$ is the set of
$m \times m$ complex matrices and a CPTP map is concretely realized as a
linear map of the form
$$\Phi: \rho \to \sum_i K_i \rho K_i^*$$
where $\rho \in M_m$ and the {\it Kraus matrices} $K_i$ are a finite family
of $n \times m$
matrices satisfying $\sum K_i^*K_i = I_m$ (the $m \times m$ identity matrix).

Positive matrices with unit trace represent mixed states, and
pure states appear as the special case of matrices of the form
$|\alpha\rangle\langle\alpha|$ for $|\alpha\rangle$ a unit vector in
${\mathbb C}^m$. Thus quantum channels transform mixed
states to mixed states, and in error correction problems one is interested
in determining which input states can be distinguished with certainty
after passing through the channel. The condition that the images of two
pure states $|\alpha\rangle$ and $|\beta\rangle$ after
transmission must be orthogonal can be expressed as
$$\langle\alpha|B|\beta\rangle = 0\qquad\mbox{ for all }
B \in \mathcal{V}_\Phi = {\rm span}\{K_i^*K_j\}$$
(see, e.g., \cite{DSW}).

The space $\mathcal{V}_\Phi \subseteq M_m$ is an {\it operator system} ---
a linear subspace of $M_m$ which is stable under the adjoint operation and
contains the identity matrix (since we have assumed that
$\sum K_i^*K_i = I_m$). In the analogous classical setting one would be
dealing with a finite set of possible input states and one could create a
graph by placing an edge between any pair of input states which might, after
transmission through a noisy channel, be received as the same output state.
This classical {\it confusability graph} is relevant to classical zero-error
communication in something like the way that the operator system
$\mathcal{V}_\Phi$ is relevant to zero-error communication in the quantum
setting. This led Duan, Severini, and Winter to term $\mathcal{V}_\Phi$ a
{\it noncommutative confusability graph} \cite{DSW}.

Going further, since any operator system can arise in the above manner
from a quantum channel, they suggested that operator systems generally
could be considered ``noncommutative graphs''. This daring proposal was
supported by the fact that they were able to define a ``quantum Lov\'{a}sz
$\vartheta$ function'' for any operator system, in analogy to the classical
Lov\'{a}sz $\vartheta$ function of a graph.

\section{Quantum relations on $M_m$}

At around the same time, the notion of a ``quantum relation''
was introduced in \cite{W}. This notion gives rise to natural
definitions of such things as ``quantum equivalence relations'' and
``quantum partial orders'', and it is also the basis of the ``quantum
metrics'' and ``quantum uniform structures'' which were studied in \cite{KW},
an earlier project from which the notion of a quantum relation was extracted.

The identification of operator systems as ``quantum graphs'' was also made
in \cite{W}, but not pursued further there.\footnote{The expression
``quantum graph'' unhappily conflicts with an earlier, unrelated use
of this term, and also with the ``noncommutative graph'' terminology
used in \cite{DSW}. But in a setting that also includes quantum
relations, quantum metrics, and so on, it is still the simple and
obvious choice.} However, it is worth investigating this connection,
as under the quantum relations point of view basic aspects of the
theory of zero-error quantum communication become conceptually transparent.

The core idea is that an {\it operator space} --- a linear subspace of
$M_m$ --- can be thought of as a quantum analog of a relation on a finite
set. (In infinite dimensions, this becomes a weak* closed operator space,
but I will stick to the finite-dimensional setting in this paper.) Classically,
a relation on a set $X$ is a subset $R$ of $X \times X$, and we write $xRy$ to
indicate that the pair $(x,y)$ belongs to the relation. The relation $R$ is
said to be
\begin{itemize}
\item {\it reflexive} if $xRx$, for all $x \in X$

\item {\it symmetric} if $xRy$ implies $yRx$, for all $x,y \in X$

\item {\it antisymmetric} if $xRy$ and $yRx$ imply $x = y$, for all
$x,y \in X$

\item {\it transitive} if $xRy$ and $yRz$ imply $xRz$, for all $x,y, z \in X$.
\end{itemize}
This can be expressed more algebraically by letting $\Delta$ be the
{\it diagonal} relation $\Delta = \{(x,x): x \in X\}$, letting
$R^t$ be the {\it transpose} relation $R^t = \{(y,x): (x,y) \in R\}$, and
letting $RR'$ be the {\it product} relation $RR' = \{(x,z): (x,y) \in R$
and $(y,z) \in R'$ for some $y \in X\}$. We can then say that $R$ is
\begin{itemize}
\item reflexive if $\Delta \subseteq R$

\item symmetric if $R = R^t$

\item antisymmetric if $R \cap R^t \subseteq \Delta$

\item transitive if $R^2 \subseteq R$.
\end{itemize}

The analogous definitions for an operator space
$\mathcal{V} \subseteq M_m$ characterize $\mathcal{V}$ as
\begin{itemize}
\item {\it reflexive} if $I_m \in \mathcal{V}$

\item {\it symmetric} if $\mathcal{V} = \mathcal{V}^*$

\item {\it antisymmetric} if $\mathcal{V} \cap \mathcal{V}^*
\subseteq \mathbb{C}\cdot I_m$

\item {\it transitive} if $\mathcal{V}^2 \subseteq \mathcal{V}$.
\end{itemize}
We define $\mathbb{C}\cdot I_m$ to be the {\it diagonal} quantum relation
on $M_m$, so that $\mathcal{V}$ is reflexive if and only if
$\mathbb{C}\cdot I_m \subseteq \mathcal{V}$, in closer analogy with the
classical case. In the above, $\mathcal{V}^* = \{A^*: A \in \mathcal{V}\}$
is the set of Hermitian adjoints of matrices in $\mathcal{V}$ and
$\mathcal{V}^2 = {\rm span}\{AB: A, B \in \mathcal{V}\}$ is a special
case of the product of two operator spaces.

Graphs appear in this framework by regarding a classical graph
as a set of vertices equipped with a reflexive, symmetric relation. The
elements of the relation represent edges, and symmetry expresses the
fact that edges are undirected. Reflexivity corresponds to the convention
that there is a loop at each vertex. This makes sense in the error
correction setting: if we place an edge between any two states which
might be confused, then it is natural to include an edge between any
state and itself. (More pointedly, it is unnatural, and creates unnecessary
complication, not to do this.) Of course, in other settings we may
not wish to impose this requirement, in which case we could drop
reflexivity and define a quantum graph to merely be a symmetric quantum
relation. This was the approach taken in \cite{S}. For the sake of
definiteness, I will use the term {\it quantum graph} to mean a reflexive,
symmetric quantum relation, i.e., an operator system, as in \cite{DSW} and
\cite{W}; however, the main results of this paper apply to quantum relations
generally, and hence also to noncommutative graphs in the broader sense of
\cite{S}.

\section{Restrictions, pushforwards, and pullbacks}

In the quantum relations setting there are natural notions of
restriction, pushforward, and pullback. Suppose we are given quantum
relations on $M_m$ and $M_n$, i.e., 
linear subspaces $\mathcal{V} \subseteq M_m$ and
$\mathcal{W} \subseteq M_n$. If $E$ is any projection in $M_m$, meaning
that $E = E^2 = E^*$, and $\Phi: M_m \to M_n$ is a CPTP map expressed as
$\Phi(\rho) = \sum_{i = 1}^d K_i\rho K_i^*$, then we define
\begin{itemize}
\item the {\it restriction} of $\mathcal{V}$ to $EM_mE$ to be
$E\mathcal{V}E = \{EAE: A \in \mathcal{V}\}$

\item the {\it pushforward} of $\mathcal{V}$ along $\Phi$ to be
\begin{eqnarray*}
\overrightarrow{\mathcal{V}} &=& \sum K_i\mathcal{V}K_j^*\cr
&=& {\rm span}\{K_iAK_j^*: A \in \mathcal{V},\,
1 \leq i,j \leq d\} \subseteq M_n
\end{eqnarray*}

\item the {\it pullback} of $\mathcal{W}$ along $\Phi$ to be
\begin{eqnarray*}
\overleftarrow{\mathcal{W}} &=& \sum K_i^*\mathcal{W}K_j\cr
&=& {\rm span}\{K_i^*BK_j: B \in \mathcal{W},\,
1\leq i,j \leq d\} \subseteq M_m.
\end{eqnarray*}
\end{itemize}
If ${\rm rank}(E) = r$ then $EM_mE$ can be identified with $M_r$, and
$E\mathcal{V}E$ with a linear subspace of $M_r$, so that the restriction of
$\mathcal{V}$ can be regarded as a quantum relation on a smaller space.

These definitions are simple and concrete. It is easy to check that if
$\mathcal{V}$ and $\mathcal{W}$ are quantum graphs (i.e., operator systems)
then so are $E\mathcal{V}E \subseteq M_r$, $\overrightarrow{\mathcal{V}}
\subseteq M_n$, and $\overleftarrow{\mathcal{W}} \subseteq M_m$.
However, the Kraus matrices $K_i$ are not uniquely determined by the map
$\Phi$ and it is not immediately apparent that the definitions of
$\overrightarrow{\mathcal{V}}$ and $\overleftarrow{\mathcal{W}}$
are independent of this choice. The definitions are also rather unmotivated.
For instance, when $\mathcal{V}$ is a quantum graph its restriction is to be
thought of as analogous to the induced subgraph construction in classical
graph theory. But an induced subgraph is obtained by choosing a subset of
the vertex set and throwing out all edges which extend out of this subset,
whereas our definition of restriction involves compressing everything in
$\mathcal{V}$ to the range of $E$. So the validity of the analogy is
unclear.

These concerns will be addressed in the next section, when we discuss
how the quantum relations point of view leads to the definitions given
above. But first, let us explain how these operations relate
to error correction.

Consider first the classical setting in which a channel is modelled by a
probabilistic transition from an initial set of states $S$ to a final set of
states $T$. That is, each initial state has some (possibly zero) probability
of going to each of the final states. Such a transition is represented by a
stochastic matrix. The confusability graph is specified by placing an
edge between two initial states if there exists a final state
to which they each have a nonzero probability of transitioning.

As I mentioned earlier, since each initial state can certainly end up at the
same state as itself, it is natural to include a loop at each vertex in this
graph. A {\it code} in this classical setting is then an independent subset
of $S$, i.e., a set of vertices with the property that the induced
subgraph contains only loops, with no edges between distinct vertices. In
the terminology of Section 2, the induced subgraph is diagonal. The
quantum analog of this would be a projection $E$ with the property that
the induced quantum subgraph $E\mathcal{V}E$ is diagonal, i.e.,
$E\mathcal{V}E = \mathbb{C}\cdot E$. If $\mathcal{V}_\Phi$ is the quantum
confusability graph mentioned in Section 1, this statement exactly expresses
the Knill-Laflamme error correction conditions \cite{KL}. So
\bigskip

{\narrower\noindent
\it the statement that the range of $E$ is a quantum code is equivalent to
the statement that $E$ induces a diagonal quantum subgraph, just as in the
classical case a code is a subset of the confusability graph for which the
induced subgraph is diagonal.
\par}
\bigskip

A more sophisticated way to specify the classical
confusability graph for a probabilistic transition from $S$ to $T$ is
to say that it is the pullback of the diagonal relation on $T$. Here we
define the pullback to $S$ of a graph on $T$ by putting an edge between two
elements of $S$ if they have a nonzero probability of mapping to adjacent
elements of $T$. The quantum analog of this construction would be
the pullback along a CPTP map $\Phi: M_m \to M_n$ of the diagonal
quantum relation on $M_n$. According to the definition of quantum
pullback given above, this would be
$${\rm span}\{K_i^*K_j: 1\leq i,j \leq d\} \subseteq M_m,$$
which is exactly the quantum confusability graph $\mathcal{V}_\Phi$.
That is,
\bigskip

{\narrower
\noindent \it the quantum confusability graph $\mathcal{V}_\Phi$ associated
to a CPTP map $\Phi: M_m \to M_n$ is the pullback along $\Phi$ of the diagonal
quantum relation on $M_n$, just as the clasical confusability graph
associated to a classical channel is the pullback of the diagonal relation.
\par}
\bigskip

The passage of a message through sequential channels provides a simple
illustration of the
value of the pullback construction. Suppose we are given classical channels
from $S$ to $T$ and from $T$ to $U$. Then their composition defines a channel
from $S$ to $U$, and the confusability graph of this composition is
the pullback to $S$ of the confusability graph for the $T$-to-$U$ channel.
In other words, it is the pullback of the pullback of the diagonal relation on
$U$. The same statement can be made in the quantum setting, as one can see
by a short computation.

A similar construction is the pushforward of the diagonal relation on $S$.
This ``dual'' confusability graph classically
includes an edge between two states in $T$ if they might have originated
in the same state of $S$. It might be used by the recipient of a signal
which was sent through a noisy channel without the aid of a code, as a
way to keep track of possible ambiguity. (This could also be a model of
a noisy measurement process in which nature is the ``sender''.) The
quantum analog would simply be the
quantum pushforward of the diagonal quantum relation.

Pushforwards and pullbacks give rise to notions of ``morphism''. Namely,
we may regard $\Phi$ as a morphism from $\mathcal{V}$ to $\mathcal{W}$
if $\overrightarrow{\mathcal{V}} \subseteq \mathcal{W}$, or, alternatively,
if $\mathcal{V} \subseteq \overleftarrow{\mathcal{W}}$. These two conditions
are not equivalent, even in the classical case: the classical analog of the
first says that any possible targets of two adjacent vertices in $S$ must
be adjacent in $T$, while the second says that any two adjacent vertices
in $S$ must have some possible targets which are adjacent in $T$. The
quantum version of the first, stronger, condition is identical to
Stahlke's notion of ``noncommutative graph homomorphism'' described in
\cite{S}:
\bigskip

{\narrower
\noindent \it a CPTP map $\Phi: M_m \to M_n$ is a ``noncommutative graph
homomorphism'' \cite{S} between operator systems $\mathcal{V} \subseteq M_m$
and $\mathcal{W} \subseteq M_n$ if $\overrightarrow{\mathcal{V}} \subseteq
\mathcal{W}$.
\par}

\section{Connecting states}

Now let us see why the definitions of restrictions, pushforwards, and
pullbacks given above are natural. The idea is to think of elements of
an operator space as ``connecting'' states. We could say that two pure states
$|\alpha\rangle, |\beta\rangle \in \mathbb{C}^m$ are connected
by a quantum relation $\mathcal{V} \subseteq M_m$ if
$\langle\alpha|B|\beta\rangle \neq 0$ for some
$B \in \mathcal{V}$. However, quantum relations are not
determined by this kind of information. For example, take $\mathcal{V}_1$
to be the set of $2\times 2$ matrices of the form $\left[
\begin{matrix}a&b\cr c&a\end{matrix}
\right]$
with $a, b, c \in \mathbb{C}$ and take $\mathcal{V}_2$ to be the full
$2\times 2$ matrix algebra $M_2$. These are both quantum graphs on $M_2$,
i.e., operator systems. It is routine to check that
$|\alpha\rangle,|\beta\rangle \in \mathbb{C}^2$ are connected by
$\mathcal{V}_1$ if and only if
neither of them is the zero vector if and only if they are connected by
$\mathcal{V}_2$. Thus, $\mathcal{V}_1$ and $\mathcal{V}_2$ are distinct
quantum relations which connect the same pairs of states in $\mathbb{C}^2$.

We must instead consider states not in $\mathbb{C}^m$ but in
$\mathbb{C}^m\otimes\mathbb{C}^k \cong \mathbb{C}^{mk}$. That is,
we consider states of a composite system
formed from the original system and some other system. Then we can define
$|\alpha\rangle,|\beta\rangle \in \mathbb{C}^{mk}$ to be connected by
$\mathcal{V}$ if there exists $B \in \mathcal{V}$ such that
$$\langle \alpha|(B\otimes I_k)|\beta\rangle \neq 0.$$
It is not hard to show that $\mathcal{V}$ is indeed determined by which
pairs of states it connects in $\mathbb{C}^{mk}$ for arbitrary $k$;
indeed, $k = m$ suffices. See Lemma \ref{keylemma} below.

It is convenient to also consider mixed states. First of all, observe that
$$|\alpha\rangle\langle \alpha|(B\otimes I_k)|\beta\rangle\langle\beta|$$
is nonzero if and only if the scalar factor
$\langle \alpha|(B\otimes I_k)|\beta\rangle$ is nonzero. So we can also say
that $\mathcal{V}$ connects $|\alpha\rangle$ and $|\beta\rangle$ if and only
if the preceding expresion is nonzero for some $B \in \mathcal{V}$. More
generally, say that $\mathcal{V}$ connects positive matrices
$A, C \in M_{mk} \cong M_m\otimes M_k$ with unit trace if
$$A(B\otimes I_k)C \neq 0$$
for some $B \in \mathcal{V}$.

Since $\mathcal{V}$ is already determined by the pairs of (composite) pure
states that it connects, it is certainly determined by the pairs of mixed
states that it connects. Indeed, any positive matrix can be expressed as a
sum of positive rank one matrices, so once we know which pure states are
connected by $\mathcal{V}$, we also know which mixed state are connected.

This point of view makes the constructions described in
the last section transparent. Let $k$ be a natural number, let
$\mathcal{V} \subseteq M_m$ and
$\mathcal{W} \subseteq M_n$ be quantum relations, let $E \in M_m$ be a
rank $r$ projection, and let $\Phi: M_m \to M_n$ be a CPTP map. Then
\begin{itemize}
\item $E\mathcal{V}E$ connects mixed states
$A, C \in EM_mE\otimes M_k \cong M_r\otimes M_k$ if and only if
$\mathcal{V}$ connects them

\item $\overrightarrow{\mathcal{V}}$ connects mixed states
$A, C \in M_n\otimes M_k$ if and only if $\mathcal{V}$ connects
$\Phi^*(A)$ and $\Phi^*(C)$

\item $\overleftarrow{\mathcal{W}}$ connects mixed states
$A, C \in M_m \otimes M_k$ if and only if $\mathcal{W}$ connects
$\Phi(A)$ and $\Phi(C)$.
\end{itemize}
See Proposition \ref{restrictchar}, Theorem \ref{concretepb}, and
Theorem \ref{concretepf} below.

Informally, the mixed states that $E\mathcal{V}E$ connects are just
the mixed states that live on $E$ and are connected by $\mathcal{V}$.
This jibes better with the ``induced subgraph'' intuition: in order to
restrict $\mathcal{V}$ to $E$, look at the pairs of states that are
connected by $\mathcal{V}$, and throw out any of them which do not lie
under $E$.

Pushforward and pullback are also easily understood in terms of connection.
The states connected by $\overrightarrow{\mathcal{V}}$ are just the states
whose images under $\Phi^*$ are connected by $\mathcal{V}$, and the states
connected by $\overleftarrow{\mathcal{W}}$ are just the states whose images
under $\Phi$ are connected by $\mathcal{W}$. This characterization shows
that the definitions of pushforward and pullback only depend on the map
$\Phi$, not the choice of Kraus matrices.

The whole point of the noncommutative confusability graph is that it
connects mixed states $A$ and $C$ if and only if $\Phi(A)\Phi(C) \neq 0$,
i.e., their image states could be confused. That is the same as saying
that their image states are connected by the diagonal quantum relation.

We can also use the idea of connecting mixed states to give an intrinsic
characterization of the ``noncommutative (directed) bipartite graphs'' of
Duan, Severini, and Winter \cite{DSW}. Given a CPTP map $\Phi: M_m \to M_n$
with Kraus matrices $K_i$, they defined this to be the span of the matrices
$K_i$. This span is no longer an operator system in general, but it is
still an operator space and hence still counts, in our terminology, as a
quantum relation. Its intrinsic characterization is simple: if
$\mathcal{V} = {\rm span}\{K_i\} \subseteq M_{n,m}$, then for any mixed
states $A \in M_m\otimes M_k$ and $C \in M_n\otimes M_k$, we have
$C(B\otimes I_k)A \neq 0$ for some $B \in \mathcal{V}$ if and only if
$\Phi(A)C \neq 0$ (letting $\Phi$ act entrywise on matrices in
$M_m\otimes M_k \cong M_k(M_m)$). That is,
\bigskip

{\narrower
\noindent \it the noncommutative bipartite graph associated to a CPTP map
$\Phi: M_m \to M_n$ connects mixed states $A \in M_m\otimes M_k$ and
$C \in M_n \otimes M_k$ if and only if $\Phi(A)C \neq 0$, i.e., there is
a possibility of confusing the image of $A$ with $C$.
\par}
\bigskip

\noindent One direction is trivial: if $C(B\otimes I_k)A = 0$
for all $B \in \mathcal{V}$, then in particular $C(K_i\otimes I_k)A = 0$
for all $i$; multiplying on the right by $(K_i^*\otimes I_k)$ and summing
over $i$ then yields $C\Phi(A) = 0$. This is equivalent to $\Phi(A)C = 0$
since $\Phi(A)$ and $C$ are positive. The reverse direction follows from
Lemma \ref{poslem} (cf.\ the proof of Theorem \ref{concretepb}).

\section{General quantum relations}

The definition of ``quantum relations'' given in \cite{W} was more general
than the one described above and actually encompasses both the classical
and quantum cases. By placing the notions of channel, confusability graph,
code, etc., in this context we obtain a common generalization in which the
classical and quantum cases are not merely analogous, but literally special
cases of a single theory. This material will be presented more formally, with
proofs of most results.

Let $\mathcal{M}$ be a unital $*$-subalgebra of $M_m$.
(In \cite{W} it could be an arbitrary von Neumann algebra in infinite
dimensions.) The two most important cases to keep in mind are
$\mathcal{M} = M_m$, the full matrix algebra, and $\mathcal{M} = D_m$,
the subalgebra of diagonal matrices. However, other cases could arise in
the presence of superselection rules; see the appendix.

Let
$$\mathcal{M}' = \{B \in M_m: AB = BA\mbox{ for all }A \in \mathcal{M}\}$$
be the commutant of $\mathcal{M}$. The commutant of $M_m$ is the scalar
algebra $\mathbb{C}\cdot I_m$, and the commutant of $D_m$ is itself. Von
Neumann's double commutant theorem states that $\mathcal{M}'' = \mathcal{M}$
always holds.

\begin{defi}\label{basicdef}
(\cite{W}, Definition 2.1)
A {\it quantum relation} on $\mathcal{M}$ is an $\mathcal{M}'$-$\mathcal{M}'$
bimodule, i.e., a linear subspace $\mathcal{V} \subseteq M_m$ satisfying
$\mathcal{M}'\mathcal{V}\mathcal{M}' \subseteq \mathcal{V}$.
\end{defi}

Here we use the operator space product $\mathcal{V}\mathcal{W} =
{\rm span}\{AB: A \in \mathcal{V}, B \in \mathcal{W}\}$.

\begin{defi}\label{def2}
(\cite{W}, Definition 2.4)
$\mathcal{M}'$ is the {\it diagonal} quantum relation on $\mathcal{M}$.
A quantum relation $\mathcal{V}$ is
\begin{itemize}
\item {\it reflexive} if $\mathcal{M}' \subseteq \mathcal{V}$

\item {\it symmetric} if $\mathcal{V}^* = \mathcal{V}$

\item {\it antisymmetric} if $\mathcal{V} \cap \mathcal{V}^* \subseteq
\mathcal{M}'$

\item {\it transitive} if $\mathcal{V}^2 \subseteq \mathcal{V}$.
\end{itemize}
\end{defi}

If $\mathcal{M} = M_m$ then $\mathcal{M}'$ is just the set of scalar matrices,
$\mathcal{M}' = \mathbb{C}\cdot I_m$,
and so any linear subspace of $M_m$ counts as a quantum relation on $M_m$
according to Definition \ref{basicdef}.
At the other extreme, it is not hard to check that quantum relations on
$D_m$ have a very transparent form. Let $E_{ij}$ be the $m\times m$ matrix
with a $1$ in the $(i,j)$ entry and $0$'s elsewhere.

\begin{prop}\label{prop1}
(\cite{W}, Proposition 2.2) If $R$ is any subset of
$\{(i,j): 1 \leq i,j \leq m\}$ then
$$\mathcal{V}_R = {\rm span}\{E_{ij}: (i,j) \in R\} \subseteq M_m$$
is a quantum relation on $D_m$, i.e., a $D_m$-$D_m$ bimodule, and every
quantum relation on $D_m$ equals $\mathcal{V}_R$ for some $R$. This
establishes a 1-1 correspondence between the classical relations on the
set $\{1, \ldots, m\}$ and the quantum relations on $D_m$.
\end{prop}

This simple result explains the justification for the term ``quantum
relation'' and also shows the value of letting $\mathcal{M}$ be any unital
$*$-subalgebra of $M_m$, not just $M_m$ itself. By taking this step we produce
a common generalization of both the classical ($\mathcal{M} = D_m$) and
elementary quantum ($\mathcal{M} = M_m$) cases. The following result
shows that the terminology of Definition \ref{def2} legitimately generalizes
the classical case.

\begin{prop}\label{prop2}
(\cite{W}, Proposition 2.5)
In the notation of Proposition \ref{prop1}, the diagonal quantum relation
on $D_m$ is $\mathcal{V}_\Delta$ where $\Delta = \{(i,i): 1 \leq i \leq m\}$.
A classical relation $R$ on $\{1, \ldots, m\}$ is
reflexive, symmetric, antisymmetric, or transitive in the ordinary  sense
if and only if the quantum relation $\mathcal{V}_R$ has the same property
in the sense of Definition \ref{def2}.
\end{prop}

The proof is easy.

Earlier we interpreted classical graphs as sets equipped with reflexive,
symmetric relations, and defined quantum graphs to be operator systems.
Both notions are subsumed in the following definition.

\begin{defi}\label{def3}
(\cite{W}, Definition 2.6 (d))
A {\it quantum graph on $\mathcal{M}$} is a reflexive, symmetric quantum
relation on $\mathcal{M}$.
\end{defi}

In the case $\mathcal{M} = M_m$, this would just mean an operator system
in $M_m$; in the case $\mathcal{M} = D_m$, by Propositions \ref{prop1}
and \ref{prop2} it effectively becomes a classical reflexive, symmetric
relation on a set.

\section{Intrinsic characterization}

The definition of a quantum relation on a unital $*$-subalgebra $\mathcal{M}
\subseteq M_m$ given in Definition \ref{basicdef} appears to depend
on the representation of $\mathcal{M}$, i.e., on the value of $m$ and perhaps
also on the way $\mathcal{M}$, regarded as an abstract algebraic structure,
is situated in $M_m$. However, this definition is
in fact effectively representation-independent, in the following sense.
Say that $*$-algebras $\mathcal{M}$ and $\mathcal{N}$ are {\it $*$-isomorphic}
if there is a linear bijection between them that is compatible with the
product and adjoint operations.

\begin{prop}\label{repindep}
(\cite{W}, Theorem 2.7)
Let $\mathcal{M} \subseteq M_m$ and $\mathcal{N} \subseteq M_n$ be unital
$*$-subalgebras and suppose they are $*$-isomorphic. Then there is a natural
1-1 correspondence between the quantum relations on $\mathcal{M}$ and the
quantum relations on $\mathcal{N}$. This correspondence takes the diagonal
quantum relation on $\mathcal{M}$ to the diagonal quantum relation on
$\mathcal{N}$ and is compatible with
the operator space product and adjoint operations on quantum relations.
\end{prop}

To see how the correspondence works, consider the case where $n = mk$ and
$\mathcal{N} = \mathcal{M}\otimes I_k \subseteq M_m\otimes M_k \cong M_n$.
Then $\mathcal{N}' = \mathcal{M}'\otimes M_k$ and, identifying
$M_n$ with $M_m\otimes M_k$, the bimodules over $\mathcal{N}'$ in $M_n$
are precisely the sets of the form $\mathcal{V}\otimes M_k$ for $\mathcal{V}$
a bimodule over $\mathcal{M}'$ in $M_m$. The full result of Proposition
\ref{repindep} is not much harder than this special case because arbitrary
$*$-isomorphisms between von Neumann algebras are not much more
general than this.

Thus, the notion of a quantum relation on $\mathcal{M}$ is effectively
independent of the representation of $\mathcal{M}$. We therefore expect
that there should be an ``intrinsic'' characterization of them which does
not reference the ambient matrix algebra. This can be achieved using the
idea of connecting states introduced in Section 4.

At that point it was convenient to consider mixed states, since we wanted
to push forward and pull back along a CPTP map, which can convert pure
states to mixed states. For the purpose of abstract characterization, it
is better to generalize pure states, which can be identified with rank
one projections, to projections of arbitrary rank. A direct connection
between the two approaches can be made by observing that for any
positive matrices $A, C \in M_m\otimes M_k$, we have
$A(B\otimes I_k)C \neq 0$ if and only if $[A](B \otimes I_k)[C] \neq 0$,
where $[A]$ denotes the range projection of $A$, i.e., the orthogonal
projection onto the range of $A$.

\begin{prop}\label{keylemma}
(\cite{W}, Lemma 2.8)
Let $\mathcal{V}$ be a proper subspace of $M_m$ and let
$B \in M_m\setminus \mathcal{V}$.

\noindent (a) There exists a natural number $k$ and vectors
$|\alpha\rangle, |\beta\rangle \in \mathbb{C}^m\otimes\mathbb{C}^k$ such that
$$\langle\alpha|(A\otimes I_k)|\beta\rangle = 0\qquad\mbox{ for all }
A \in \mathcal{V}$$
but
$$\langle\alpha|(B \otimes I_k)|\beta\rangle \neq 0.$$

\noindent (b) If $\mathcal{M}$ is a unital $*$-subalgebra of $M_m$ and
$\mathcal{V}$ is a quantum relation on $\mathcal{M}$, then there exist
$k \in \mathbb{N}$ and projections $P, Q \in M_k(\mathcal{M})$ such that
$$P(A\otimes I_k)Q = 0\qquad\mbox{ for all }A \in \mathcal{V}$$
but
$$P(B\otimes I_k)Q \neq 0.$$
\end{prop}

\begin{proof}
(a) We prove the result with $k = m$. First, observe that $M_m$ becomes an
inner product space when equipped with the Hilbert-Schmidt inner product
$$\langle A_1,A_2\rangle = {\rm Tr}(A_1^*A_2).$$
So there must exist $C \in M_m$ satisfying
$${\rm Tr}(AC) = 0\qquad\mbox{ for all }A \in \mathcal{V}$$
but
$${\rm Tr}(BC) \neq 0.$$
Let $|c_1\rangle, \ldots, |c_m\rangle \in \mathbb{C}^m$ be the columns of $C$,
let $|e_1\rangle, \ldots, |e_m\rangle$ be
the standard basis vectors in $\mathbb{C}^m$, and let
$|\alpha\rangle, |\beta\rangle \in \mathbb{C}^m\otimes\mathbb{C}^m \cong
\mathbb{C}^m \oplus \cdots \oplus \mathbb{C}^m$ be the vectors
$$|\alpha\rangle = |e_1\rangle \oplus \cdots \oplus |e_m\rangle\qquad
|\beta\rangle = |c_1\rangle \oplus \cdots \oplus |c_m\rangle.$$
Then
$$\langle\alpha|(A \otimes I_m)|\beta\rangle
= \sum_i \langle e_i|A|c_i\rangle
= {\rm Tr}(AC) = 0$$
for all $A \in \mathcal{V}$, and similarly
$\langle\alpha|(B\otimes I_m)|\beta\rangle = {\rm Tr}(BC) \neq 0$.

(b) Find $|\alpha\rangle, |\beta\rangle \in \mathbb{C}^m\otimes\mathbb{C}^m$
as in part (a) and let
$P, Q \in M_m\otimes M_m$ be the orthogonal projections onto
$(\mathcal{M}'\otimes I_m)|\alpha\rangle
= \{(A\otimes I_m)|\alpha\rangle: A \in \mathcal{M}'\}$
and $(\mathcal{M}'\otimes I_m)|\beta\rangle
= \{(A\otimes I_m)|\beta\rangle: A \in \mathcal{M}'\}$,
respectively. Since these spaces are invariant for the $*$-algebra
$\mathcal{M}'\otimes I_m$, $P$ and $Q$ belong to its commutant
$(\mathcal{M}'\otimes I_m)' = \mathcal{M} \otimes M_m \cong M_m(\mathcal{M})$.

Now $Q|\beta\rangle = |\beta\rangle$, so the range of $(B\otimes I_m)Q$
contains the vector $(B\otimes I_m)|\beta\rangle$, which is not orthogonal
to $|\alpha\rangle$. Since $|\alpha\rangle$ belongs to
the range of $P$, it follows that $P(B\otimes I_m)Q \neq 0$. However, if
$A \in \mathcal{V}$ and $A_1, A_2 \in \mathcal{M}'$ then $A_1^*AA_2 \in
\mathcal{V}$, so that
$$\langle(A_1\otimes I_m)\alpha|(A\otimes I_m)|(A_2\otimes I_m)\beta\rangle
= \langle\alpha|(A_1^*AA_2\otimes I_m)|\beta\rangle = 0.$$
Since this is true for any $A_1, A_2 \in \mathcal{M}'$, it follows that
$\langle\alpha'|(A\otimes I_m)|\beta'\rangle = 0$ for all $|\alpha'\rangle$
and $|\beta'\rangle$ in the ranges of $P$ and $Q$, respectively. Thus
$P(A\otimes I_m)Q = 0$.
\end{proof}

Say that $\mathcal{V}$ connects projections $P, Q \in M_k(\mathcal{M})$
if there exists $A \in \mathcal{V}$ such that $P(A\otimes I_k)Q \neq 0$. The
preceding result shows that $\mathcal{V}$ is determined by the pairs of
projections it connects in this manner: we can tell whether a given
$B \in M_m$ belongs to $\mathcal{V}$ by testing whether it connects any
pair of projections that is not connected by $\mathcal{V}$. Since this is
a crucial point, let us emphasize it: if $\mathcal{M}$ is a unital
$*$-subalgebra of $M_m$ then an $\mathcal{M}'$-$\mathcal{M}'$ bimodule is
determined by the pairs of projections in $M_m(\mathcal{M})$ that it connects.

In fact, quantum relations can be characterized abstractly in these
terms. We give the relevant definition first, and then state the
equivalence with Definition \ref{basicdef} as a theorem. To avoid
confusion, we will now refer to quantum relations in the sense of
Definition \ref{basicdef} as {\it concrete} quantum relations.

Let $\mathcal{P}(M_k(\mathcal{M}))$
denote the set of projections in $M_k(\mathcal{M})$, given the topology it
inherits from $M_k(\mathcal{M})$.

\begin{defi}\label{intdef}
(\cite{W}, Definition 2.24)
Let $\mathcal{M}$ be a unital $*$-subalgebra of $M_m$, and for each
$k \in \mathbb{N}$ let $\mathcal{R}_k$ be an open subset of
$\mathcal{P}(M_k(\mathcal{M}))\times\mathcal{P}(M_k(\mathcal{M}))$. Then
the sequence $(\mathcal{R}_k)$ is
an {\it intrinsic quantum relation} if
\smallskip

{\narrower{
\noindent (i) $(0,0) \not\in \mathcal{R}_k$

\noindent (ii) $(\bigvee P_i, \bigvee Q_j) \in \mathcal{R}_k$ if and only
if $(P_{i_0}, Q_{j_0}) \in \mathcal{R}_k$ for some $i_0$, $j_0$

(iii) $(P, [BQ]) \in \mathcal{R}_k$ if and only if
$([B^*P], Q) \in \mathcal{R}_l$
\smallskip}}

\noindent for all $k,l \in \mathbb{N}$, all projections $P, P_i, Q_j \in
M_k(\mathcal{M})$ and $Q \in M_l(\mathcal{M})$, and all scalar matrices
$B \in I_m\otimes M_{k,l}$.
\end{defi}

In condition (ii) the join $\bigvee P_i$ of a finite family of projections
$(P_i)$ is defined to be the orthogonal projection onto the span of their
ranges. In (iii), recall that the notation $[B]$ indicates the range
projection of $B$.

$\mathcal{R}_k$ is to be thought of as the pairs of projections in
$M_k(\mathcal{M})$ which are connected by some concrete quantum relation
$\mathcal{V} \subseteq M_m$.
To say that each $\mathcal{R}_k$ is open is to say that if two projections
are connected then so are any two projections sufficiently close to them.
Condition (i) is trivial, condition (ii) is the basic axiom characterizing
connection, and condition (iii) is a statement about scalar compatibility
that is typical of what one sees when working at matrix levels. The point
is that if $B$ is a scalar matrix then
\begin{eqnarray*}
P(A\otimes I_k)[BQ] = 0 &\Leftrightarrow& P(A\otimes I_k)BQ = 0\cr
&\Leftrightarrow& PB(A\otimes I_l)Q = 0\cr
&\Leftrightarrow& [B^*P](A\otimes I_l)Q = 0,
\end{eqnarray*}
since $(A\otimes I_k)B = B(A\otimes I_l)$.

Proposition \ref{keylemma} (b) shows us how to go from concrete quantum
relations, as characterized by Definition \ref{basicdef}, to intrinsic
quantum relations, axiomatized as in Definition \ref{intdef}. Namely,
given $\mathcal{V}$, for each $k \in \mathbb{N}$
let $\mathcal{R}_k$ be the set of pairs $(P,Q)$
of projections in $M_k(\mathcal{M})$ such that $P(A\otimes I_k)Q \neq 0$
for some $A \in \mathcal{V}$. Conversely, given an intrinsic
quantum relation $(\mathcal{R}_k)$ one recovers the concrete quantum
relation that corresponds to it as the set of $A \in M_m$ satisfying
$$P(A \otimes I_k)Q = 0$$
for all $k \in \mathbb{N}$ and all $(P,Q) \not\in \mathcal{R}_k$, i.e.,
the set of matrices which do not connect any pair of projections they are
not supposed to connect.

\begin{theo}\label{axioms}
(\cite{W}, Theorem 2.32)
For any unital $*$-subalgebra $\mathcal{M}$ of $M_m$, the two constructions
just described establish a 1-1 correspondence between the concrete
and intrinsic quantum relations on $\mathcal{M}$.
\end{theo}

The proof of Theorem \ref{axioms} is somewhat complicated.

Observe that the characterization of quantum relations provided by
Definition \ref{intdef} is ``intrinsic'' to $\mathcal{M}$ in the sense
that it makes no reference to the ambient matrix algebra in which
$\mathcal{M}$ is located. It is manifestly compatible with $*$-isomorphisms.

\section{Restrictions}

Theorem \ref{axioms} allows us to pass back and forth between concrete
and intrinsic quantum relations, and we will do this repeatedly in the
sequel.

An $\mathcal{M}'$-$\mathcal{M}'$ bimodule is a straightforward object,
especially when $\mathcal{M} = M_m$ and $\mathcal{M}' = \mathbb{C}\cdot I_m$.
The value in having a more complicated intrinsic characterization in terms
of connecting projections is that some constructions are more naturally
understood in these terms. For instance, the natural notion of ``subobject''
is the following.

\begin{defi}\label{restrict}
Let $\mathcal{M}$ be a unital $*$-subalgebra of $M_m$, let
$(\mathcal{R}_k)$ be an intrinsic quantum relation
on $\mathcal{M}$, and let $E \in \mathcal{M}$ be a projection of rank $r$.
The {\it restriction} of $(\mathcal{R}_k)$ to
$E\mathcal{M}E \subseteq EM_mE \cong M_r$ is the intrinsic quantum relation
$(\tilde{\mathcal{R}}_k)$ on $E\mathcal{M}E$ defined by setting
$$\tilde{\mathcal{R}}_k = \{(P,Q) \in \mathcal{R}_k: P, Q \leq E\otimes I_k\}$$
for all $k \in \mathbb{N}$.
\end{defi}

It is straightforward to verify that $(\tilde{\mathcal{R}}_k)$ as defined
above has the properties of an intrinsic quantum relation described
in Definition \ref{intdef}. So according to Theorem \ref{axioms},
if $(\mathcal{R}_k)$ is associated to the concrete
quantum relation $\mathcal{V} \subseteq M_m$, its restriction
$(\tilde{\mathcal{R}}_k)$ must be associated to a concrete quantum relation
$\tilde{\mathcal{V}} \subseteq M_r$ on $E\mathcal{M}E$. This concrete
restriction has a simple direct characterization:

\begin{prop}\label{restrictchar}
Let $\mathcal{V}$ be a concrete quantum relation on $\mathcal{M}
\subseteq M_m$ and let $E \in \mathcal{M}$ be a projection. Then the
restriction $\tilde{\mathcal{V}}$ of $\mathcal{V}$ to $E\mathcal{M}E$
is concretely given as $\tilde{\mathcal{V}} = E\mathcal{V}E$.
\end{prop}

\begin{proof}
First observe that the commutant of $E\mathcal{M}E$ in $EM_mE \cong M_r$ is
$E\mathcal{M}'E$. The containment $E\mathcal{M}'E \subseteq (E\mathcal{M}E)'$
is clear because
$$(EBE)(EAE) = EBAE = EABE = (EAE)(EBE)$$
for all $A \in \mathcal{M}$ and $B \in \mathcal{M}'$,
showing that everything in $E\mathcal{M}'E$
commutes with everything in $E\mathcal{M}E$. For the reverse containment,
by the double commutant theorem it suffices to show that
$(E\mathcal{M}'E)' \subseteq E\mathcal{M}E$. So let
$A \in M_r \cong EM_mE$ belong
to the commutant of $E\mathcal{M}'E$. Regarding $A$ as an element of $M_m$
satisfying $A = EAE$, this means that $(EBE)A = A(EBE)$ for all
$B \in \mathcal{M}'$. But $BE = EB$ and $AE = EA = A$, so it follows
that $BA = AB$ for all $B \in \mathcal{M}'$, i.e., by the double commutant
theorem, $A \in \mathcal{M}$. Thus $A \in E\mathcal{M}E$, as desired.

The computations
$$(EAE)(EBE) = EABE\qquad{\rm and}\qquad (EBE)(EAE) = EBAE$$
for $A \in \mathcal{V}$ and $B \in \mathcal{M}'$ now show that
$E\mathcal{V}E$ is a bimodule over $(E\mathcal{M}E)'
= E\mathcal{M}'E$, i.e., it is a quantum relation on $E\mathcal{M}E$.

Now let $(\mathcal{R}_k)$ be the intrinsic quantum relation on $\mathcal{M}$
corresponding to $\mathcal{V}$, $(\tilde{\mathcal{R}}_k)$ the restriction of
$(\mathcal{R}_k)$ to $E\mathcal{M}E$ according to Definition \ref{restrict},
and $(\tilde{\mathcal{R}}'_k)$ the intrinsic quantum relation on
$E\mathcal{M}E$ corresponding to $E\mathcal{V}E$. We must show
that $(\tilde{\mathcal{R}}_k) = (\tilde{\mathcal{R}}'_k)$.

Fix $k \in \mathbb{N}$.
In one direction, if $(P,Q) \in \tilde{\mathcal{R}}_k'$ then there
exists $EAE \in E\mathcal{V}E$ such that
$$P(EAE\otimes I_k)Q \neq 0.$$
But since $P,Q \leq E\otimes I_k$ and
$$EAE\otimes I_k = (E\otimes I_k)(A\otimes I_k)(E\otimes I_k),$$
this implies that $P(A\otimes I_k)Q \neq 0$,
so that $(P,Q)$ belongs to $\mathcal{R}_k$ and therefore to
$\tilde{\mathcal{R}}_k$. Conversely, if $(P,Q) \in \tilde{\mathcal{R}}_k$
then $(P,Q) \in \mathcal{R}_k$ and so $P(A\otimes I_k)Q \neq 0$
for some $A \in \mathcal{V}$. But since $P,Q \leq E\otimes I_k$, this
implies that $P(EAE\otimes I_k)Q \neq 0$, and we therefore have
$(P,Q) \in \tilde{\mathcal{R}}_k'$. This completes the proof of the
desired equality.
\end{proof}

Although this concrete description of the restriction of $\mathcal{V}$ to
$E\mathcal{M}E$ is very simple, it is the intrinsic formulation given
in Definition \ref{restrict} which brings out its role as a ``restriction''.

The following definition now becomes natural.

\begin{defi}\label{independent}
Let $\mathcal{V}$ be a quantum graph (a reflexive, symmetric quantum relation)
on $\mathcal{M} \subseteq M_m$ and let $E \in \mathcal{M}$ be a projection.
Then $E$ is {\it independent} if the restriction of $\mathcal{V}$ to
$E\mathcal{M}E$ is the diagonal quantum relation on $E\mathcal{M}E$.
\end{defi}

In the case $\mathcal{M} = D_m$, the projection $E$ corresponds to a subset
of $\{1, \ldots, m\}$, and $E$ is independent in the above sense if and
only if the classical graph corresponding to $\mathcal{V}$ has no nontrivial
edges in this subset. That is, Definition \ref{independent}
generalizes the classical notion of an independent subset of a graph.
In the case $\mathcal{M} = M_m$, the independence condition simply states that
$$E\mathcal{V}E = \mathbb{C}\cdot E,$$
which, as we noted earlier, expresses the Knill-Laflamme error correction
conditions. So the notion of independence yields a common generalization
of classical and quantum codes.

\section{Pushforwards}

Classically, if $f: X \to Y$ is a function between sets then we can push
any relation $R$ on $X$ forward to a relation on $Y$, namely, the relation
$\{(f(x), f(y)): (x,y) \in R\}$. Similarly, any relation $R'$ on $Y$ can be
pulled back to the relation $\{(x,y): (f(x), f(y)) \in R'\}$ on $X$.
We now seek quantum versions of these constructions.

The first point to make is that the classical analog of a quantum channel
is not an actual function between sets, but a classical channel which maps
points in the domain to probability distributions in the range (representing
the likelihood of the given input state being received as various output
states). In this context the pushforward of a relation $R$ on $X$ would
consist of the pairs $(x',y') \in Y^2$ such that there exists a pair
$(x,y) \in R$
for which the transition probabilities $x \to x'$ and $y \to y'$ are both
nonzero. The pullback of a relation $S$ on $Y$ would consist of the pairs
$(x,y) \in X^2$ such that there exists a pair $(x',y') \in S$ for which the
transition probabilities $x \to x'$ and $y \to y'$ are both nonzero.

Since we are working with unital $*$-algebras, it is natural to adopt the
Heisenberg picture in which algebras of observables transform. Mathematically,
this means that instead of the CPTP map $\Phi: \rho \mapsto \sum K_i\rho K_i^*$
from $M_m$ to $M_n$ mentioned in Section 1, which acts on states, we consider
the adjoint map $\Psi: A \mapsto \sum K_i^*AK_i$ from $M_n$ to $M_m$, which
acts on observables. The adjoint of a CPTP map is a unital CP (unital
completely positive) map. Taking adjoints reverses arrows, so that
pushforwards become pullbacks and vice versa; consequently, to maintain
consistency with Section 3 we will continue to take the CPTP map
$\Phi: \mathcal{M} \to \mathcal{N}$ as primary, even though at this point
it becomes less natural. The map $\Phi$ really should be understood as a
map from the predual of
$\mathcal{M}$ to the predual of $\mathcal{N}$ whose adjoint unital CP map
$\Psi = \Phi^*$ takes the $*$-algebra $\mathcal{N}$ to the $*$-algebra
$\mathcal{M}$, but any finite-dimensional $*$-algebra can be identified with
its predual via the pairing
$(A, B) \mapsto {\rm Tr}(AB)$, so we need not make this distinction.

The unital CP maps which correspond to actual functions between sets are the
unital $*$-homomorphisms, linear maps $\Psi: \mathcal{N} \to \mathcal{M}$
which preserve the identity and respect the product and adjoint operations.
If $\Psi = \Phi^*$ is a $*$-homomorphism and
$(\mathcal{R}_k)$ is an intrinsic quantum relation on $\mathcal{M}$, there is
an obvious way to push forward a quantum relation $(\mathcal{R}_k)$ on
$\mathcal{M}$ along $\Phi$ to a quantum relation $(\tilde{\mathcal{R}}_k)$
on $\mathcal{N}$. Namely, for each $k \in \mathbb{N}$ let
$\tilde{\mathcal{R}}_k$ consist of those pairs of projections
$P, Q \in M_k(\mathcal{N})$ with the property that
$(\Psi(P), \Psi(Q)) \in \mathcal{R}_k$. (Here we abuse notation and also denote
by $\Psi$ the map from $M_k(\mathcal{N})$ to $M_k(\mathcal{M})$ which applies
$\Psi$ entrywise.) This definition makes sense because the
$*$-homomorphism property ensures that $\Psi(P)$ and $\Psi(Q)$ are projections.
It is easy to check that the preceding construction does yield an intrinsic
quantum relation on $\mathcal{N}$ (\cite{W}, Proposition 2.25 (b)).

But we are interested in general CPTP maps, not just those whose adjoint
maps are $*$-homomorphisms. The construction just described no longer works
because the image of a projection under such a map need not be a projection.
However, there is a simple solution to this difficulty. Let us consider
two positive matrices $A, B \in M_n$ to be equivalent if $[A] = [B]$.
Since the range of a Hermitian matrix is the orthocomplement of its kernel,
this condition could also be stated as ${\rm ker}(A) = {\rm ker}(B)$.
This notion of equivalence is suitable here because whether positive matrices
are connected by a quantum relation depends only on their range projections.

\begin{lemma}\label{basiclemma}
Let $\mathcal{M}$ and $\mathcal{N}$ be unital $*$-subalgebras of $M_m$ and
$M_n$, respectively, let $A, B \in \mathcal{N}$ be positive, and let
$\Psi: \mathcal{N} \to \mathcal{M}$ be a CP map.
Then $[A] = [B]$ implies $[\Psi(A)] = [\Psi(B)]$.
\end{lemma}

\begin{proof}

Recall that the join $P \vee Q$ of two projections $P$ and $Q$ is the
orthogonal projection onto the span of their ranges. We first claim
that $[A+B] = [A] \vee [B]$. (We are not yet assuming $[A] = [B]$, only
that $A$ and $B$ are both positive.)
That is, we claim that ${\rm ran}(A + B) = {\rm ran}(A) + {\rm ran}(B)$.
The containment $\subseteq$ is clear. Conversely, suppose $|\alpha\rangle \perp
{\rm ran}(A + B)$, i.e., $|\alpha\rangle \in {\rm ker}(A + B)$. Then
$$0 = \langle\alpha|(A + B)|\alpha\rangle
= \langle\alpha|A|\alpha\rangle + \langle\alpha|B|\alpha\rangle.$$
Since $A$ and $B$ are positive, this implies that
$\langle\alpha|A|\alpha\rangle = \langle\alpha|B|\alpha\rangle = 0$ and
therefore that $A|\alpha\rangle = B|\alpha\rangle = 0$. So
$|\alpha\rangle \perp {\rm ran}(A)$ and $|\alpha\rangle \perp {\rm ran}(B)$,
and therefore $|\alpha\rangle \perp {\rm ran}(A) + {\rm ran}(B)$. This shows
that ${\rm ran}(A) + {\rm ran}(B) \subseteq {\rm ran}(A + B)$, and so the
first claim is proven.

Now assume $[A] = [B]$. We next claim that $[K^*AK] = [K^*BK]$ for any
$n\times m$ matrix $K$. To see this, let $|\alpha\rangle \in {\rm ker}(K^*AK)$.
Then $\langle\alpha|K^*AK|\alpha\rangle = 0$, that is,
$\langle A^{1/2}K\alpha|A^{1/2}K\alpha\rangle = 0$, and this implies that
$K|\alpha\rangle \in {\rm ker}(A^{1/2}) = {\rm ker}(A)$. Since $[A] = [B]$,
we get $K|\alpha\rangle \in {\rm ker}(B)$, and therefore
$|\alpha\rangle \in {\rm ker}(K^*BK)$. So we
have shown that ${\rm ker}(K^*AK) \subseteq {\rm ker}(K^*BK)$. By
symmetry the reverse containment also holds, so we conclude that the
two kernels are equal, i.e., $[K^*AK] = [K^*BK]$.

We can now prove the lemma.
We have $\Psi(C) = \sum K_i^*CK_i$ for some finite family of
$n \times m$ matrices $K_i$. So, using the two claims, we have
$$[\Psi(A)] = \left[\sum K_i^*AK_i\right] = \bigvee [K_i^*AK_i] =
\bigvee [K_i^*BK_i] = \left[\sum K_i^*BK_i\right] = [\Psi(B)],$$
as desired.
\end{proof}

We note that a version of Lemma \ref{basiclemma} for
normal CP maps between von Neumann algebras can be proven using
the normal Stinespring theorem (\cite{Bl}, Theorem III.2.2.4).

We can now describe the appropriate version of the pushforward construction
for CP maps. Here we return to the ``connecting mixed states'' point
of view; note that if $A, C \in \mathcal{M}\otimes M_k$ are positive then
since $\Psi$ is completely positive, $\Psi(A)$ and
$\Psi(C)$ will also be positive. Lemma \ref{basiclemma} shows that CP maps
preserve the relevant notion of equivalence between positive matrices.

\begin{defi}\label{pushforward}
Let $\mathcal{M} \subseteq M_m$ and $\mathcal{N} \subseteq M_n$ be
unital $*$-subalgebras and let $\Phi: \mathcal{M} \to \mathcal{N}$ be
a CP map. Suppose $(\mathcal{R}_k)$ is an intrinsic quantum relation on
$\mathcal{M}$. Then its {\it pushforward along $\Phi$} is the intrinsic
quantum relation
$(\overrightarrow{\mathcal{R}}_k)$ on $\mathcal{N}$ defined by, for each
$k \in \mathbb{N}$, letting $(P,Q)$ belong to $\overrightarrow{\mathcal{R}}_k$
if $([\Phi^*(P)], [\Phi^*(Q)])$ belongs to $\mathcal{R}_k$.
\end{defi}

In order to justify this definition, we must check that
$(\overrightarrow{\mathcal{R}}_k)$ satisfies the axioms given in Definition
\ref{intdef}. This can be done directly using Lemma \ref{basiclemma}, but
according to Theorem \ref{axioms}, it can also be done by finding
a concrete quantum relation $\mathcal{W}$ on $\mathcal{N}$ with the property
that $(P,Q) \in \overleftarrow{\mathcal{R}}_k$ if and only if
$P(A\otimes I_k)Q \neq 0$ for some $A \in \mathcal{W}$. This will be
achieved in Theorem \ref{concretepb} below. Thus, that theorem will
simultaneously establish that the pushforward construction is well-defined
and identify its concrete formulation.

We require two simple facts about positive matrices.

\begin{lemma}\label{poslem}
(a) Let $A, C \in M_m$ be positive, let $B \in M_n$, and
let $K_1, K_2 \in M_{m,n}$. Then $K_1^*AK_1BK_2^*CK_2 = 0$ if and only
if $AK_1BK_2^*C = 0$.

\noindent (b) Let $A, X_i, Y_j \in M_m$ and suppose the $X_i$ and $Y_j$ are
positive. Then $(\sum X_i)A(\sum Y_j) = 0$ if and only if $X_iAY_j = 0$ for
all $i$ and $j$.
\end{lemma}

\begin{proof}
(a) The reverse implication is trivial. For the forward implication
let $D = BK_2^*CK_2$ and suppose $K_1^*AK_1D = 0$. Then
$$0 = D^*K_1^*AK_1D = (A^{1/2}K_1D)^*(A^{1/2}K_1D),$$
so $A^{1/2}K_1D = 0$ and therefore $AK_1D = 0$, i.e., $AK_1BK_2^*CK_2 = 0$.
Applying the same argument to the adjoint of the expression
$AK_1BK_2^*CK_2$ then yields the conclusion $AK_1BK_2^*C = 0$.

(b) Again, the reverse implication is trivial. For the forward implication,
we claim that if $(X_1 + X_2)B = 0$ with $X_1, X_2 \geq 0$ then
$X_1B = X_2B = 0$. This inductively implies the same statement with
any finite number of $X_i$'s. Taking $B = A(\sum Y_j)$ in the statement
of the lemma then yields $X_iA(\sum Y_j) = 0$ for all $i$, and applying
the same argument to the adjoint of each of these expressions produces
the desired conclusion.

To verify the claim, suppose $(X_1 + X_2)B = 0$. Then
$$0 = B^*(X_1 + X_2)B = B^*X_1B + B^*X_2B,$$
and since both $B^*X_1B$ and $B^*X_2B$ are positive, this implies
that both are zero. It follows that $X_1B = X_2B = 0$, as claimed.
\end{proof}

\begin{theo}\label{concretepb}
Let $\mathcal{M} \subseteq M_m$ and $\mathcal{N} \subseteq M_n$ be
unital $*$-subalgebras and let $\Phi: \mathcal{M} \to \mathcal{N}$ be
a CP map given by $\Phi: B \mapsto \sum_{i=1}^d K_iBK_i^*$. Suppose
$\mathcal{V} \subseteq M_m$ is a concrete quantum relation on
$\mathcal{M}$. Then its pushforward is concretely given as the
$\mathcal{N}'$-$\mathcal{N}'$ bimodule generated by
$$\{K_iAK_j^*: A \in \mathcal{V},\, 1 \leq i,j \leq d\}.$$
\end{theo}

\begin{proof}
Let $\mathcal{W}$ be the $\mathcal{N}'$-$\mathcal{N}'$ bimodule generated by
the matrices $K_iAK_j^*$ for $A \in \mathcal{V}$ and $1 \leq i,j \leq d$. We
must show that for any $k \in \mathbb{N}$ and any projections
$P,Q \in M_k(\mathcal{N})$, we have $P(B\otimes I_k)Q \neq 0$ for some
$B \in \mathcal{W}$ if and only if
$[\Phi^*(P)](A\otimes I_k)[\Phi^*(Q)] \neq 0$
for some $A \in \mathcal{V}$.

Since $P$ and $Q$ commute with anything in $\mathcal{N}'\times I_k$,
the condition
$$P(B\otimes I_k)Q \neq 0\mbox{ for some }B \in \mathcal{W}$$
obtains if and only if
$$P(K_iAK_j^*\otimes I_k)Q \neq 0\mbox{ for some }A \in \mathcal{V}
\mbox{ and some }i,j.$$
Equivalently,
$$P(K_i\otimes I_k)(A\otimes I_k)(K_j^*\otimes I_k)Q \neq 0$$
for some $A \in \mathcal{V}$ and some $i,j$, which by Lemma \ref{poslem} (a)
is equivalent to
$$(K_i^*\otimes I_k)P(K_i\otimes I_k)(A\otimes I_k)
(K_j^*\otimes I_k)Q(K_j^*\otimes I_k) \neq 0$$
for some $A \in \mathcal{V}$ and some $i,j$. Then
since $\Phi^*(P) = \sum (K_i^*\otimes I_k)P(K_i\otimes I_k)$ and
$\Phi^*(Q) = \sum (K_j^*\otimes I_k)Q(K_j\otimes I_k)$, by Lemma \ref{poslem}
(b) the last statement is equivalent to
$$\Phi^*(P)(A\otimes I_k)\Phi^*(Q) \neq 0\mbox{ for some }A \in \mathcal{V},$$
which is trivially equivalent to
$$[\Phi^*(P)](A\otimes I_k)[\Phi^*(Q)] \neq 0\mbox{ for some }A \in \mathcal{V},$$
as desired.
\end{proof}

If $\mathcal{N} = M_n$ then its commutant is the set of scalar matrices,
so that the pushforward described in Theorem \ref{concretepb} is just the
linear span of the matrices $K_iAK_j^*$.

Once we know how to push forward quantum relations, it is easy to say what
the appropriate notion of ``morphism'' should be: if $\mathcal{M}$ and
$\mathcal{N}$ are both equipped with quantum relations $\mathcal{V}$ and
$\mathcal{W}$, then a CPTP map from $\mathcal{M}$ to $\mathcal{N}$ should be
considered a morphism if the pushforward $\overrightarrow{\mathcal{V}}$ of
$\mathcal{V}$ is contained in $\mathcal{W}$. The classical version (which is
recovered as the case where $\mathcal{M} = D_m$ and $\mathcal{N} = D_n$)
would be a classical channel from a set $S$ of size $m$ to a set
$T$ of size $n$ for which the pushforward of a given relation on
$S$ is contained in a given relation on $T$.

Various definitions of quantum graph homomorphisms were proposed in
\cite{OP, RM, S}. Here the term ``homomorphism''
conflicts somewhat with classical usage, where a homomorphism between graphs
is usually taken to be an actual function between the vertex sets, not a
channel which could map vertices to probability distributions.
An actual map between classical sets generalizes in the quantum setting to a
$*$-homomorphism from $\mathcal{N}$ to $\mathcal{M}$.
In the more general setting of CP maps we prefer the term ``CP morphism'':

\begin{defi}\label{cpmorphism}
Let $\mathcal{M} \subseteq M_m$ and $\mathcal{N} \subseteq M_n$ be
unital $*$-subalgebras equipped with intrinsic quantum relations
$(\mathcal{R}_k)$ and $(\mathcal{S}_k)$, respectively. A {\it CP morphism}
from $\mathcal{M}$ to $\mathcal{N}$ is then a CP map
$\Phi: \mathcal{M} \to \mathcal{N}$ with the property that
$\overrightarrow{\mathcal{R}}_k \subseteq \mathcal{S}_k$ for all $k$.
\end{defi}

In terms of concrete quantum relations $\mathcal{V}$ and $\mathcal{W}$ on
$\mathcal{M}$ and $\mathcal{N}$, respectively, the condition would be that
$\overrightarrow{\mathcal{V}} \subseteq \mathcal{W}$, where
$\overrightarrow{\mathcal{V}}$ is the pushforward of $\mathcal{V}$.
In particular, if $\mathcal{M}$ and $\mathcal{N}$ are matrix algebras
and $\mathcal{V}$ and $\mathcal{W}$ are quantum graphs (i.e., operator
systems), the concrete formulation given in Theorem \ref{concretepb} states
that the condition for $\Phi$ to be a CP morphism is
$K_i\mathcal{V}K_j^* \subseteq \mathcal{W}$ for all $i$ and $j$, which is
Stahlke's condition \cite{S}. However, the
$\overrightarrow{\mathcal{R}}_k \subseteq \mathcal{S}_k$ forulation is
manifestly intrinsic.

\section{Pullbacks}

There is also a way to pull quantum relations back via a CP map.
Since we already know how to push quantum relations forward, one
obvious solution is just to push forward using the adjoint map.
This makes perfect sense in the finite-dimensional setting, but it
fails in infinite dimensions when von Neumann algebras can no longer
be identified with their preduals. However, there is an alternative
approach which does generalize to infinite dimensions.
We describe this construction now.

The key question is how to use a CP map $\Psi = \Phi^*: \mathcal{N} \to \mathcal{M}$
to turn a projection in $\mathcal{M}$ into a projection in $\mathcal{N}$.
We can do this using hereditary cones. A {\it hereditary cone} in $\mathcal{M}$
is a nonempty set $\mathcal{C}$ of positive matrices in $\mathcal{M}$
with the properties
\smallskip

{\narrower{
\noindent (i) if $A \in \mathcal{C}$ then $aA \in \mathcal{C}$ for all
$a \geq 0$

\noindent (ii) if $A, B \in \mathcal{C}$ then $A + B \in \mathcal{C}$

(iii) if $A \in \mathcal{C}$ and $0 \leq B \leq A$ then
$B \in \mathcal{C}$.
\smallskip}}

\noindent If $P$ is a projection in $\mathcal{M}$ then $\mathcal{C}_P =
\{A \in \mathcal{M}: A \geq 0$ and $PA = 0\}$ is a hereditary cone, and
it is not hard to check that every hereditary cone in $\mathcal{M}$ has
this form. As it is easy to see that the inverse image under any CP
map $\Psi: \mathcal{N} \to \mathcal{M}$ of a hereditary cone in $\mathcal{M}$
is a hereditary cone in $\mathcal{N}$, this shows us how to use $\Psi$ to
turn a projection $P \in \mathcal{M}$ into a projection
$\overleftarrow{\Psi}(P)$ in $\mathcal{N}$: take $\overleftarrow{\Psi}(P)$
to satisfy $\mathcal{C}_{\overleftarrow{\Psi}(P)} = \Psi^{-1}(\mathcal{C}_P)$.

We can now define the pushforward of a quantum relation via a CP map. We
continue to abuse notation by letting $\overleftarrow{\Psi}$ also denote the
matrix level map which takes projections in $M_k(\mathcal{M})$ to projections
in $M_k(\mathcal{N})$.

\begin{defi}\label{pullback}
Let $\mathcal{M} \subseteq M_m$ and $\mathcal{N} \subseteq M_n$ be
unital $*$-subalgebras and let $\Phi: \mathcal{M} \to \mathcal{N}$ be
a CP map. Suppose $(\mathcal{S}_k)$ is an intrinsic quantum relation on
$\mathcal{N}$. Then its {\it pullback} is the intrinsic quantum relation
$(\overleftarrow{\mathcal{S}}_k)$ on $\mathcal{M}$ defined by, for each
$k \in \mathbb{N}$, letting $(P,Q)$ belong to $\overleftarrow{\mathcal{S}}_k$
if $(\overleftarrow{\Psi}(P), \overleftarrow{\Psi}(Q))$
belongs to $\mathcal{S}_k$, where $\Psi = \Phi^*$.
\end{defi}

As with pushforwards, we must justify this definition by showing that
$(\overleftarrow{\mathcal{S}}_k)$ satisfies the axioms for an intrinsic
quantum relation, and as in that case we will accomplish this by
identifying the concrete quantum relation that corresponds to
$(\overleftarrow{\mathcal{S}}_k)$.

\begin{theo}\label{concretepf}
Let $\mathcal{M} \subseteq M_m$ and $\mathcal{N} \subseteq M_n$ be
unital $*$-subalgebras and let $\Phi: \mathcal{M} \to \mathcal{N}$ be
a CP map given by $\Phi: B \mapsto \sum_{i=1}^d K_iBK_i^*$. Suppose
$\mathcal{W} \subseteq M_n$ is a concrete quantum relation on
$\mathcal{N}$. Then its pullback is concretely given as the
$\mathcal{M}'$-$\mathcal{M}'$ bimodule generated by
$$\{K_i^*BK_j: B \in \mathcal{W},\, 1 \leq i,j \leq d\}.$$
\end{theo}

\begin{proof}
Fix $k \in \mathbb{N}$ and $P,Q \in M_k(\mathcal{M})$. Let $\Psi = \Phi^*$.
We first claim that
$\overleftarrow{\Psi}(P) = \bigvee [(K_i\otimes I_k)P]$. To see this,
recall that a positive matrix $A \in M_k(\mathcal{N})$ belongs to
$\mathcal{C}_{\overleftarrow{\Psi}(P)}$ if and only if $\Psi(A)P = 0$.
But $\Psi(A) = \sum (K_i^*\otimes I_k)A(K_i\otimes I_k)$, so by Lemma
\ref{poslem} we have $\Psi(A)P = 0$ if and only if $A(K_i\otimes I_k)P = 0$
for all $i$. The claim now follows from the definition of
$\mathcal{C}_{\overleftarrow{\Psi}(P)}$.

Now the condition for $(P,Q)$ to belong to the pullback of the
intrinsic quantum relation associated to $\mathcal{W}$ is that
$$\overleftarrow{\Psi}(P)(B\otimes I_k)
\overleftarrow{\Psi}(Q) \neq 0$$
for some $B \in \mathcal{W}$. By the claim and Lemma \ref{poslem} (b),
this happens if and only if
$$[(K_i\otimes I_k)P](B\otimes I_k)[(K_j\otimes I_k)Q] \neq 0$$
for some $B \in \mathcal{W}$ and some $i,j$. This equivalent to saying that
$$P(K_i^*\otimes I_k)(B\otimes I_k)(K_j\otimes I_k)Q \neq 0,$$
i.e.,
$$P(K_i^*BK_j\otimes I_k)Q \neq 0,$$
for some $B \in \mathcal{W}$. Since $P$ and $Q$ commute with $A\otimes I_k$
for every $A \in \mathcal{M}'$, this last condition is equivalent to the
statement that $(P,Q)$ belongs to the intrinsic quantum relation associated
to the $\mathcal{M}'$-$\mathcal{M}'$-bimodule generated by the matrices
$K_i^*BK_j$. We conclude that the latter bimodule is the concrete form of
the pullback of $\mathcal{W}$.
\end{proof}

Again, in the case where $\mathcal{M} = M_m$, this pullback would
simply be the linear span of the matrices $K_i^*BK_j$.

The pullback construction gives rise to an alternative version of CP
morphism which is weaker than the one proposed in Definition \ref{cpmorphism}.
Namely, instead of requiring
$\overrightarrow{\mathcal{R}}_k \subseteq \mathcal{S}_k$ for all $k$ we
could require $\mathcal{R}_k \subseteq \overleftarrow{\mathcal{S}}_k$ for
all $k$. In concrete terms, the condition that $K_i\mathcal{V}K_j^*
\subseteq \mathcal{W}$ for all $i$ and $j$ is replaced by the condition
that $\mathcal{V} \subseteq \sum_{i,j}K_i^*\mathcal{W}K_j$. The second
condition is implied by the first (multiply the first condition on the
left by $K_i^*$ and on the right by $K_j$, then sum over $i$ and $j$ and
invoke the identity $\sum K_i^*K_i = I_m$). Classically, the first version
demands that if the point $x$ is related to the point
$y$ then $x'$ must be related to $y'$ for any $x'$ and $y'$ such that
the transition probabilities $x \to x'$ and $y \to y'$ are both nonzero,
while the second version asks only that there be at least one such pair
$(x',y')$.

\section{Appendix: superselection rules}

A central feature of quantum mechanics is the possibility of forming
superpositions of states, famously illustrated by the parable of
Schrodinger's cat. However, not all superpositions are physically
allowed. For instance, in elementary quantum mechanics one cannot
prepare an isolated system in a superposition of two states in which different
numbers of particles are present, or whose total charges are different. Such
restrictions are known as ``superselection rules''. In their presence the
Hilbert space of the system will decompose into an orthogonal sum of
``superselection sectors'', subspaces within which all pairs of states
can be superposed.

In particular, one cannot perform a measurement on an isolated system which
could result in the system being in a forbidden superposition. This means that
all physical observables are restricted to have a block diagonal form such
that every eigenspace lies within some superselection sector.
Thus, the physical observables should typically be modelled not by arbitrary
self-adjoint matrices in $M_m$, but rather by self-adjoint matrices belonging
to a fixed unital $*$-subalgebra of $M_m$. The significance of superselection
rules in quantum information theory has been studied in \cite{BW, JWBVP, KMP}.

A qualification is in order here. When one is encoding information in a
quantum mechanical system, superselection rules are only absolute if the
system is isolated. Thus, in principle, one could achieve a ``forbidden''
superposition by first coupling the system of interest to an external system,
then preparing the composite system in the desired superposition, then
discarding the external system without measuring it. That is, if there is
a rule forbidding superposition of the states $|\alpha\rangle$ and
$|\alpha'\rangle$ in the first system when it is isolated, one finds states
$|\beta\rangle$ and $|\beta'\rangle$ of the second system
such that superposition of $|\alpha\rangle|\beta\rangle$ and
$|\alpha'\rangle|\beta'\rangle$ is not forbidden. However, even if possible
in principle, this strategy may not be feasible in practice. Indeed, the
whole reason why ordinary communication is classical is exactly because one is
practically unable to create quantum superpositions of macroscopic objects.
Thus, operative superselection rules may arise not from fundamental physics
but from limitations of the experimental apparatus.

For example, suppose that a sender populates a potential well with $n$
hydrogen atoms and transmits the system to a recipient, with the transmitted
information being simply the value of $n$. If the system is isolated then
states with different values of $n$ definitely cannot be superposed, but even
if this restriction is relaxed the sender might not have the ability to create
such superpositions. If so, this information is classical.

To make this example more interesting, suppose the sender has the ability to
prepare the atoms in desired spin states and wants to encode information in
this way. States of the system with different values of $n$ still cannot be
superposed, but states with the same value of $n$ can be. In this case the
relevant $*$-algebra $\mathcal{M}$ is a direct sum
$\bigoplus M_{2^n}$ because there are $2^n$ possible basic spin states
in a system with $n$ atoms.

Another way that $\mathcal{M}$ could be neither $M_m$ nor $D_m$ is if the
transmitted information comes in two parts, a classical part and a quantum
part. In that case we would have $\mathcal{M} = M_m\otimes D_k$ where $m$
is the number of degrees of freedom of the quantum part of the message and
$k$ is the number of degrees of freedom of the classical part. A situation
of this type could arise if, say, information is encoded in an array of
heavy atoms using both their spin states and their centers of mass.
(The $i$th atom might be placed at one of two locations $(i,1)$ or $(i,2)$
in a $k\times 2$ grid, for example.) Models in which the spin of an atom is
treated quantum mechanically while its center of mass is treated classically
are familiar from standard analyses of the Stern-Gerlach experiment.

How would basic notions from quantum error correction generalize to the
mixed setting? Given the abstract formulations presented earlier, this
question is easy to answer. The situation would be that we have unital
$*$-algebras $\mathcal{M} \subseteq M_m$ and $\mathcal{N} \subseteq M_n$
and a CPTP map $\Phi: \mathcal{M} \to \mathcal{N}$. The quantum confusability
graph would then be the pullback along $\Phi$ of the diagonal quantum relation
on $\mathcal{N}$, i.e., its commutant $\mathcal{N}'$. If $\Phi$ has the
form $\Phi: B \mapsto \sum_{i=1}^d K_iBK_i^*$, then according to Theorem
\ref{concretepf} this pullback would be the $\mathcal{M}'-\mathcal{M}'$
bimodule generated by $\{K_i^*BK_j: B \in \mathcal{N}'\}$.

As a simple special case, suppose $\mathcal{M} = M_m\otimes D_{m'} \subseteq
M_{mm'}$ and $\mathcal{N} = M_n\otimes D_{n'} \subseteq M_{nn'}$. Then
$\mathcal{M}' = I_m\otimes D_{m'}$ and $\mathcal{N}' = I_n\otimes D_{n'}$,
and the quantum confusability graph is the $\mathcal{M}'-\mathcal{M}'$
bimodule generated by $\{K_i^*(I_n\otimes B)K_j: B \in D_{n'}\}$. This
expression can be made more explicit by if the Kraus matrices are chosen
in a natural way. Namely, for each $1 \leq a \leq m$ and $1 \leq b \leq n$
the map $\Phi$ induces a CP map from $M_m \cong M_m\otimes E_{aa} \subseteq
M_{mm'}$ to $M_n \cong M_n\otimes E_{bb} \subseteq M_{nn'}$, where $E_{aa}$
is the $m'\times m'$ matrix with a $1$ in the $(a,a)$ entry and $0$'s
elsewhere and $E_{bb}$ is the $n'\times n'$ matrix with a $1$ in the
$(b,b)$ entry and $0$'s elsewhere. Thus, for each $a$ and $b$ we can
find a family of $n\times m$ Kraus matrices $K_i^{ab}$ such that $\Phi(B)
= \sum_{a,b,i} (K_i^{ab}\otimes E_{aa})B(K_i^{ab}\otimes E_{aa})^*\otimes
E_{bb}$, and the condition that $\Phi$ be trace preserving is
$\sum_{a,b,i} (K_i^{ab})^*K_i^{ab}\otimes E_{aa} = I_{mm'}$, or equivalently,
$\sum_{b,i} (K_i^{ab})^*K_i^{ab} = I_m$ for each $a$ (a version of
stochasticity). With these Kraus
matrices, the quantum confusability graph is the operator system $\mathcal{V}$
spanned by the matrices $(K_i^{ab})^*K_j^{a'b}\otimes E_{aa'} \in M_{mm'}$ for
arbitrary $a$, $a'$, $b$, $i$, and $j$,
which is automatically an $\mathcal{M}'-\mathcal{M}'$ bimodule.

In this setting a quantum code would be a projection $E \in \mathcal{M}$
such that $E\mathcal{V}E = I_m\otimes D_{m'}$, where $\mathcal{V}$ is the
quantum confusability graph just described. Concretely, we can write
$E = \sum P_a\otimes E_{aa}$ where each $P_a$ is a projection in $M_m$, and
the error correction conditions would state that (1)
for each $a$, $b$, $i$, and $j$ the matrix $P_a(K_i^{ab})^*K_j^{ab}P_a$ is a
scalar multiple of $P_a$, and (2) for each $a \neq a'$, $b$, $i$, and $j$ the
matrix $P_a(K_i^{ab})^*K_j^{a'b}P_{a'}$ is zero.


\end{document}